\documentclass[oneside,english]{amsart}
\usepackage[latin9]{inputenc}
\usepackage{amstext}
\usepackage{amsthm}
\usepackage{amssymb}
\usepackage{stackrel}
\PassOptionsToPackage{normalem}{ulem}
\usepackage{ulem}

\makeatletter

\DeclareTextSymbolDefault{\textquotedbl}{T1}

\numberwithin{equation}{section}
\numberwithin{figure}{section}
\theoremstyle{plain}
\newtheorem{thm}{\protect\theoremname}
\theoremstyle{definition}
\newtheorem{defn}[thm]{\protect\definitionname}

\makeatother

\usepackage{babel}
\providecommand{\definitionname}{Definition}
\providecommand{\theoremname}{Theorem}

\begin{document}
\title{{\normalsize{}Abundance of arithmetic progressions in some combiantorially
large sets }}
\author{Pintu Debnath and Sayan Goswami}
\address{{\large{}Pintu Debnath, Department of Mathematics, Basirhat College,
Basirhat -743412, North 24th parganas, West Bengal, India.}}
\email{{\large{}pintumath1989@gmail.com}}
\address{{\large{}Sayan Goswami, Department of Mathematics, University of Kalyani,
Kalyani-741235, Nadia, West Bengal, India.}}
\email{{\large{}sayan92m@gmail.com}}
\thanks{{\large{}The second author is supported by UGC-JRF fellowship.}}
\begin{abstract}
{\large{}Furstenberg and Glasner proved that for an arbitrary $k\in\mathbb{N}$,
any piecewise syndetic set contains $k-$term arithmetic progression
and such collection is also piecewise syndetic in $\mathbb{Z}.$ They
used algebraic structure of $\beta\mathbb{N}$. The above result was
extended for arbitrary semigroups by Bergelson and Hindman, again
using the structure of Stone-\v{C}ech compactification of general
semigroup. However they provided the abundances for various types
of large sets. But the abundances in many large sets is still unknown.
In this work we will provide the abundance in quasi-central sets,
$J-sets$ and $C-sets$.}{\large\par}
\end{abstract}

\maketitle

\section{introduction}

{\large{}A subset $S$ of $\mathbb{\ensuremath{Z}}$ is called syndetic
if there exists $r\in\mathbb{N}$ such that $\bigcup_{i=1}^{r}(S-i)=\mathbb{Z}$.
Again a subset $S$ of $\mathbb{\ensuremath{Z}}$ is called thick
if it contains arbitrary long intervals in it. Sets which can be expressed
as intersection of thick and syndetic sets, are called piecewise syndetic. }{\large\par}

{\large{}One of the famous Ramsey theoretic results is so called van
derWaerden\textquoteright s Theorem \cite{key-9} which states that
one cell of any partition $\{C_{1},C_{2},\ldots,C_{r}\}$ of $\mathbb{N}$
contains arithmetic progression of arbitrary length. Since arithmetic
progressions are invariant under shifts, it follows that every piecewise
syndetic set contains arbitrarily long arithmetic progressions.}{\large\par}

{\large{}Furstenberg and E. Glasner in \cite{key-3} algebraically
and Beiglboeck in \cite{key-1.1} combinatorially, proved that if
$S$ is a piecewise syndetic subset of $\mathbb{Z}$ and $l\in\mathbb{N}$
then the set of all length $l$ progressions contained in $S$, is
also large.}{\large\par}
\begin{thm}
{\large{}\label{Thm 1} Let $k\in\mathbb{N}$ and assume that $S\subseteq\mathbb{Z}$
is piecewise syndetic. Then $\{(a,d)\,:\,a,a+d,\ldots,a+kd\in S\}$
is piecewise syndetic in $\mathbb{Z}^{2}$.}{\large\par}
\end{thm}

{\large{}The above theorem can be proved for the set of Natural Numbers
$\mathbb{N}$ in a similar way. In \cite{key-2,key-7} the above result
was studied for various large sets viz. Central, Thick, IP sets etc.
for general semigroups. But for there are many large sets that are
remained to be studied. }{\large\par}

{\large{}To state about those sets we need some prerequisite of Stone-\v{C}ech
compactification of general semigroup which is given below.}{\large\par}

{\large{}Let $\beta S$, be the ultrafilters on $S$, identifying
the principal ultrafilters with the points of $S$ and thus pretending
that $S\subseteq\beta S$. Given$A\subseteq S$ let us set, 
\[
\overline{A}=\{p\in\beta S:A\in p\}.
\]
 Then the set $\{\overline{A}:A\subseteq S\}$ is a basis for a topology
on $\beta S$. The operation$+$ on $S$ can be extended to the Stone-\v{C}ech
compactification $\beta S$ of $S$ so that $(\beta S,+)$ is a compact
right topological semigroup (meaning that for any $p\in\beta S$,
the function $\rho_{p}:\beta S\rightarrow\beta S$ defined by $\rho_{p}(q)=q+p$
is continuous) with $S$ contained in its topological center (meaning
that for any $x\in S$, the function $\lambda_{x}:\beta S\rightarrow\beta S$
defined by $\lambda_{x}(q)=x+q$ is continuous). Given $p,q\in\beta S$
and $A\subseteq S$, $A\in p+q$ if and only if $\{x\in S:-x+A\in q\}\in p$,
where $-x+A=\{y\in S:x+y\in A\}$. }{\large\par}

{\large{}A nonempty subset $I$ of a semigroup $(T,+)$ is called
a left ideal of $\emph{T}$ if $T+I\subset I$, a right ideal if $I+T\subset I$,
and a two sided ideal (or simply an ideal) if it is both a left and
right ideal. A minimal left ideal is the left ideal that does not
contain any proper left ideal. Similarly, we can define minimal right
ideal and smallest ideal.}{\large\par}

{\large{}Any compact Hausdorff right topological semigroup $(T,+)$
has a smallest two sided ideal}{\large\par}

{\large{}
\[
\begin{array}{ccc}
K(T) & = & \bigcup\{L:L\text{ is a minimal left ideal of }T\}\\
 & = & \,\,\,\,\,\bigcup\{R:R\text{ is a minimal right ideal of }T\}
\end{array}
\]
}{\large\par}

{\large{}Given a minimal left ideal $L$ and a minimal right ideal
$R$, $L\cap R$ is a group, and in particular contains an idempotent.
An idempotent in $K(T)$ is called a minimal idempotent. If $p$ and
$q$ are idempotents in $T$, we write $p\leq q$ if and only if $p+q=q+p=p$.
An idempotent is minimal with respect to this relation if and only
if it is a member of the smallest ideal. }{\large\par}
\begin{defn}
{\large{}A set $A\subseteq S$ in a semigroup $(S,\cdot)$ is said
to be an IP-set if $A$ belongs to some idempotent of $\beta S$.
A set $E\subset S$ is called called an IP$^{*}$-set iff it meets
nontrivially every IP-set, alternatively if $E$ is contained in every
idempotent in $\beta S$.}{\large\par}
\end{defn}

{\large{}The definition of central set in \cite{key-2-1} was in terms
of dynamical systems, and the definition makes sense in any semigroup.
In \cite{key-1-1} that definition was shown to be equivalent to a
much simpler algebraic characterization if the semigroup is countable.
It is this algebraic characterization which we take as the definition
for all semigroups,}{\large\par}
\begin{defn}
{\large{}Let $(S,.)$ be a semigroup and let $A\subseteq S$. Then
$A$ is central if and only if there is some minimal idempotent $p\in\beta S$
with $p\in\overline{A}.$}{\large\par}
\end{defn}

{\large{}However the two sided ideal $K(\beta S)$ is not closed in
$\beta S$ and any member of any idempotent in the closure of $K(\beta S)$
is called Quasi-central sets.}{\large\par}
\begin{defn}
{\large{}\cite[Definition 1.2]{key-5} Let $(S,.)$ be a semigroup
and let $A\subseteq S$. Then $A$ is quasi-central if and only if
there is some idempotent $p\in cl(K(\beta S))$ with $p\in\overline{A}.$}{\large\par}
\end{defn}

{\large{}It has a nice combinatorial property which is given:}{\large\par}
\begin{thm}
{\large{}\label{Th 1} \cite[Theorem 3.7]{key-5} For a countable
semigroup $(S,.)$, $A\subseteq S$ is said to be Quasi-central iff
there is a decreasing sequence $\langle C_{n}\rangle_{n=1}^{\infty}$
of subsets of $A$ such that,}{\large\par}

{\large{}$(1)$ \label{prop 1} for each $n\in\mathbb{N}$ and each
$x\in C_{n}$, there exists $m\in\mathbb{N}$ with $C_{m}\subseteq x^{-1}C_{n}$
and}{\large\par}

{\large{}$(2)$ \label{prop 2} $C_{n}$ is piecewise syndetic $\forall n\in\mathbb{N}$.}{\large\par}
\end{thm}

{\large{}The importance of Quasi Central sets is it is very close
to Central Sets and enjoy a close combinatorial property to those
sets. }{\large\par}

{\large{}There is another important set which is known as $J-set$
defined as }{\large\par}
\begin{defn}
{\large{}Let $(S,+)$is a commutative semigroup and let $A\subseteq S$
is said to be a $J-set$ iff whenever $F\in\mathcal{P}_{f}\left(S^{\mathbb{N}}\right)$,
there exist $a\in S$ and $H\in\mathcal{P}_{f}(\mathbb{N})$ such
that for each $f\in F,$ $a+\sum_{t\in H}f(t)\in A$.}{\large\par}
\end{defn}

{\large{}It can be shown that a piecewise syndetic set is also a $J-set$
\cite[Theorem 14.8.3, page 336]{key-6}. The set $J(S)=\left\{ p\in\beta S:for\,all\,A\in p,A\,is\,a\,J-set\right\} $
is a compact two sided ideal of $\beta S$. The $C-set$ was defined
as those sets satiesfying the conclusion of the central set theorem
\cite[Theorem 2.2]{key-0}. It can be shown that if $A$ is $C-set$,
then there exist an idempotent $p\in J(S)$ such that $A\in p$. It
has a nice combinatorial property as theorem 5 given below:}{\large\par}
\begin{thm}
{\large{}\label{Th 2} \cite[Theorem 14.27, page 358]{key-6} For
a countable semigroup $(S,.)$, $A\subseteq S$ is a $C-set$ iff
there is a decreasing sequence $\langle C_{n}\rangle_{n=1}^{\infty}$
of subsets of $A$ such that,}{\large\par}

{\large{}$(1)$ \label{prop 2-1} for each $n\in\mathbb{N}$ and each
$x\in C_{n}$, there exists $m\in\mathbb{N}$ with $C_{m}\subseteq x^{-1}C_{n}$
and}{\large\par}

{\large{}$(2)$ \label{prop 2-2} $C_{n}$ is a $J-set$ $\forall n\in\mathbb{N}$.}{\large\par}
\end{thm}

{\large{}In \cite[Theorem 3.6]{key-5-1} it was proved that polynomial
progressions in piecewise syndetic sets are abundance in nature.}{\large\par}

{\large{}Now first we give analog result of \ref{Thm 1} for quasi
central sets, then in $J-sets$ and $C-sets$.}{\large\par}

\section{\uline{proof of main theorems}}
\begin{thm}
{\large{}\label{THM 7} For any quasi-central $M\subseteq\mathbb{N}$
the collection $\{(a,b):\,\{a,a+b,a+2b,\ldots,a+lb\}\subset M\}$
is quasi-central in $(\mathbb{N\times\mathbb{N}},+)$.}{\large\par}
\end{thm}

\begin{proof}
{\large{}As $M$ is quasi-central, there exists a decreasing sequence
piecewise syndetic subsets of $\mathbb{N}$, $\{A_{n}:n\in\mathbb{N}\}$
satisfying the property 1 in theorem 5.}{\large\par}

{\large{}As all $A_{n}$ are piecewise syndetic $\forall n\in\mathbb{N}$
in the following sequence,}{\large\par}

{\large{}\label{eqn 1}
\[
1.\quad M\supseteq A_{1}\supseteq A_{2}\supseteq\ldots\supseteq A_{n}\supseteq\ldots
\]
}{\large\par}

{\large{}The set $B=\{(a,b):\,\{a,a+b,a+2b,\ldots,a+lb\}\subset M\}$
is piecewise syndetic in $\mathbb{N\times\mathbb{N}}$ from theorem
1.}{\large\par}

{\large{}And for $i\in\mathbb{N},$ $B_{i}=\{(a,b)\in\mathbb{N\times\mathbb{N}}:\,\{a,a+b,a+2b,\ldots,a+lb\}\subset A_{i}\}\neq\phi$
is piecewise syndetic $\forall i\in\mathbb{N},$ theorem 1.}{\large\par}

{\large{}Consider,\label{eqn 2}}{\large\par}

{\large{}
\[
2.\quad B\supseteq B_{1}\supseteq B_{2}\supseteq\ldots\supseteq B_{n}\supseteq\ldots
\]
}{\large\par}

{\large{}Now choose $n\in\mathbb{N}$ and $(a,b)\in B_{n}$, then
$\{a,a+b,a+2b,\ldots,a+lb\}\subset A_{n}$.}{\large\par}

{\large{}Now choose by property 1, there exists $N\in\mathbb{N}$
such that,
\begin{equation}
A_{N}\subseteq\stackrel[i=0]{l}{\bigcap}(-(a+ib)+A_{n})
\end{equation}
}{\large\par}

{\large{}Now any $(a_{1},b_{1})\in B_{N}$ implies $\{a_{1},a_{1}+b_{1},a_{1}+2b_{1},\ldots,a_{1}+lb_{1}\}\subseteq A_{N}\subseteq\stackrel[i=0]{l}{\bigcap}(-(a+ib)+A_{n})$ }{\large\par}

{\large{}So $(a_{1}+a)+i.(b_{1}+b)\in A_{n}\forall i\in\{0,1,2,\ldots,l\}$,
hence $(a_{1},b_{1})\in-(a,b)+B_{n}$.}{\large\par}

{\large{}This implies $B_{N}\subseteq-(a,b)+B_{n}$. }{\large\par}

{\large{}Therefore for any $(a,b)\in B_{n}$, there exists $N\in\mathbb{N}$
such that $B_{N}\subseteq-(a,b)+B_{n}$ showing the property \ref{prop 1}.}{\large\par}

{\large{}This proves the theorem.}{\large\par}
\end{proof}
{\large{}Now we will give a result analog of theorem 1 for $J-sets$.}{\large\par}
\begin{thm}
{\large{}For any  $J-set$ $A\subseteq\mathbb{N}$ the collection
$\{(a,b):\,\{a,a+b,a+2b,\ldots,a+lb\}\subset A\}$ is $J-set$ in
$(\mathbb{N\times\mathbb{N}},+)$.}{\large\par}
\end{thm}

\begin{proof}
{\large{}Let $C=\left\{ \left(a,d\right):\left\{ a,a+d,....,a+ld\right\} \subseteq A\right\} \ldots\left(3\right)$
and our goal is to show $C$ is a $J-set$ in $(\mathbb{N\times\mathbb{N}},+)$.}{\large\par}

{\large{}Now any $f\in\mathbb{N}^{\mathbb{N}}$has the form $f=(f_{1},f_{2})$
where $f_{1},f_{2}:\mathbb{N}\mathbb{\longrightarrow N}$.}{\large\par}

{\large{}Considering $F\in P_{f}\left((\mathbb{N}\times\mathbb{N})^{\mathbb{N}}\right)$.
Then we will have to show there exists $\left(a_{1},a_{2}\right)\in\mathbb{N}\times\mathbb{N}$
and $H_{1}\in P_{f}\left(\mathbb{N}\right)$ such that for each $f\in F$,
$\left(a_{1},a_{2}\right)+\sum_{t\in H_{1}}f\left(t\right)\in C$.}{\large\par}

{\large{}Now, assuming $F=\left\{ f_{1},f_{2},\ldots,f_{m}\right\} $
where $f_{i}=\left(g_{2i-1},g_{2i}\right)$ for $i=1,2,\ldots,m.$
Then take any $b\in\mathbb{N}$ and consider the set $G\in P_{f}\left(\mathbb{N^{N}}\right)$,
where 
\[
G=\:\left\{ g_{1},g_{1}+i_{1}\left(b+g_{2}\right),\ldots,g_{2m-1},g_{2m-1}+i_{m}\left(b+g_{2m}\right)\right\} _{i_{1},i_{2},\ldots,i_{m}=1}^{l}
\]
and given $A$ is a J-set, we get $a\in\mathbb{N},\:H\in P_{f}\left(\mathbb{N}\right)$
such that $a+\sum_{t\in H}h\left(t\right)\in A$ for all $h\in G$,
i.e. $a+\sum_{t\in H}\left(g_{2i-1}+j\left(b+g_{2i}\right)\right)\left(t\right)\in A,\text{ where }i=1,2,\ldots,m\:\&\:j=0,1,....,l$.}{\large\par}

{\large{}i.e., $a+\sum_{t\in H}g_{2i-1}\left(t\right)+j\left(b\left|H\right|+\sum_{t\in H}g_{2i}\left(t\right)\right)\in A$}{\large\par}

{\large{}i.e., $\left(a+\sum_{t\in H}g_{2i-1}\left(t\right),b\left|H\right|+\sum_{t\in H}g_{2i}\left(t\right)\right)\in C$
(Follows from $\left(3\right)$)}{\large\par}

{\large{}i.e., $\left(a,b\left|H\right|\right)+\sum_{t\in H}\left(g_{2i-1},g_{2i}\right)\left(t\right)\in C$}{\large\par}

{\large{}i.e., $\left(a,b\left|H\right|\right)+\sum_{t\in H}f_{i}\left(t\right)\in C$
where $i=1,2,\ldots,m$.}{\large\par}

{\large{}So, for any choosen $F$ in $P_{f}\left((\mathbb{N}\times\mathbb{N})^{\mathbb{N}}\right)$,
there exist $\left(a,b\left|H\right|\right)\in\mathbb{N}\times\mathbb{N}$
and $H\in P_{f}\left(\mathbb{N}\right)$ such that $\left(a,b\left|H\right|\right)+\sum_{t\in H}f_{i}\left(t\right)\in C,\:f_{i}\in F,\:i=1,2,3$.}{\large\par}
\end{proof}
{\large{}Now we will give an analogous version of theorem 8.}{\large\par}
\begin{thm}
{\large{}For any  $C-set$ $A\subseteq\mathbb{N}$ the collection
$\{(a,b):\,\{a,a+b,a+2b,\ldots,a+lb\}\subset A\}$ is $C-set$ in
$(\mathbb{N\times\mathbb{N}},+)$}{\large\par}
\end{thm}

\begin{proof}
{\large{}As $A$ is $C-set$, there exists a decreasing sequence $J-sets$
in $\mathbb{N}$, $\{A_{n}:n\in\mathbb{N}\}$ satisfying the property
1 in theorem 7.}{\large\par}

{\large{}As all $A_{n}$ are $J-set$ $\forall n\in\mathbb{N}$ in
the following sequence,}{\large\par}

{\large{}\label{eqn 4}
\[
4.\quad A\supseteq A_{1}\supseteq A_{2}\supseteq\ldots\supseteq A_{n}\supseteq\ldots
\]
}{\large\par}

{\large{}The set $B=\{(a,b):\,\{a,a+b,a+2b,\ldots,a+lb\}\subset M\}$
is $J-set$ in $\mathbb{N\times\mathbb{N}}$ from theorem 9.}{\large\par}

{\large{}And for $i\in\mathbb{N},$ $B_{i}=\{(a,b)\in\mathbb{N\times\mathbb{N}}:\,\{a,a+b,a+2b,\ldots,a+lb\}\subset A_{i}\}\neq\phi$
are $J-sets$ $\forall i\in\mathbb{N},$ theorem 9.}{\large\par}

{\large{}Consider,\label{eqn 5}}{\large\par}

{\large{}
\[
5.\quad B\supseteq B_{1}\supseteq B_{2}\supseteq\ldots\supseteq B_{n}\supseteq\ldots
\]
}{\large\par}

{\large{}Now choose $n\in\mathbb{N}$ and $(a,b)\in B_{n}$, then
$\{a,a+b,a+2b,\ldots,a+lb\}\subset A_{n}$.}{\large\par}

{\large{}Now by property 1, of theorem 7, there exists $N\in\mathbb{N}$
such that,
\begin{equation}
A_{N}\subseteq\stackrel[i=0]{l}{\bigcap}(-(a+ib)+A_{n})
\end{equation}
}{\large\par}

{\large{}Now any $(a_{1},b_{1})\in B_{N}$ implies $\{a_{1},a_{1}+b_{1},a_{1}+2b_{1},\ldots,a_{1}+lb_{1}\}\subseteq A_{N}\subseteq\stackrel[i=0]{l}{\bigcap}(-(a+ib)+A_{n})$ }{\large\par}

{\large{}So $(a_{1}+a)+i.(b_{1}+b)\in A_{n}\forall i\in\{0,1,2,\ldots,l\}$,
hence $(a_{1},b_{1})\in-(a,b)+B_{n}$.}{\large\par}

{\large{}This implies $B_{N}\subseteq-(a,b)+B_{n}$. }{\large\par}

{\large{}Therefore for any $(a,b)\in B_{n}$, there exists $N\in\mathbb{N}$
such that $B_{N}\subseteq-(a,b)+B_{n}$ showing the property 1 of
theorem 7.}{\large\par}

{\large{}This proves the theorem.}{\large\par}

\textbf{\large{}Acknowledgment:}{\large{} The second author acknowledges
the UGC NET-JRF grant.}{\large\par}
\end{proof}

\end{document}